\documentclass{llncs}  
\usepackage[affil-it]{authblk}
\pdfoutput=1
\usepackage[english]{babel}
\usepackage[IL2]{fontenc} 
\usepackage[utf8]{inputenc}   

\usepackage{amsmath}
\usepackage{amsfonts} 
\allowdisplaybreaks

\usepackage{subcaption}
\usepackage{algorithm}
\usepackage{algpseudocode}

\usepackage{graphicx}

\usepackage{tabularx} 
\usepackage{dcolumn} 

\usepackage{url} 

\usepackage{epstopdf} 

\usepackage{caption}

\usepackage{float}

\numberwithin{equation}{section}

\newtheorem{mydef}{Definition}
\newtheorem{prop}{Proposition}
\newtheorem{coroll}{Corollary}

\newcommand{\jed}[1]{\ensuremath{~\mathrm{#1}}} 

\newcommand{\m}[1]{\ensuremath{\mathbf{#1}}} 

\newcommand{\cmax}{\ensuremath{C_{\text{max}}}}

\newcommand{\hi}{\ensuremath{I_\mathcal{MC}|_{\mathcal{X}_i = 3}}}
\newcommand{\midlvl}{\ensuremath{I_\mathcal{MC}|_{\mathcal{X}_j = 2}}}
\newcommand{\lo}{\ensuremath{I_\mathcal{MC}|_{\mathcal{X}_k = 1}}}

\begin{document}
\title{On Solving Non-preemptive Mixed-criticality Match-up Scheduling Problem with Two and Three Criticality Levels}
\date{}
\author{Antonin Novak\inst{1,2}\footnote{The corresponding author: \texttt{novakan9@fel.cvut.cz}}, Premysl Sucha\inst{2}, Zdenek Hanzalek\inst{1,2}}
\institute{Czech Institute of Informatics, Robotics and Cybernetics, \\Czech Technical University in Prague, CZ \and Department of Control Engineering, Faculty of Electrical Engineering, \\Czech Technical University in Prague, CZ}
\maketitle
\begin{abstract}
In this paper, we study an \mbox{$\mathcal{NP}$-hard} problem of a single machine scheduling minimizing the makespan, where the mixed-critical tasks with uncertain processing time are scheduled. We show the derivation of \mbox{F-shaped} tasks from the probability distribution function of the processing time, then we study the structure of problems with two and three criticality levels for which we propose efficient exact algorithms and we present computational experiments for instances with up to $200$ tasks. Moreover, we show that the considered problem is approximable within a constant multiplicative factor.
\end{abstract}



\section{INTRODUCTION} \label{sec:intro}

The communication buses in modern vehicles are an essential part of advanced driver assistants. Those systems depend on the data gather by sensors, such as LIDAR, cameras, and radars. The data about the surrounding environment are communicated through communication buses to ECUs (\textit{Electronic Control Units}) where they are processed, and appropriate actions are taken. For example, if an obstacle is detected in front of the vehicle, the car starts break automatically. Not only driver assistants rely on the communication. Different ECUs are responsible for running car as a whole. The fuel is injected accordingly to the current combustion and outside conditions, mod- ern cars with drive-by-wire system steer via electronic signals, even the windows are controlled by the central infotainment system.

The modern vehicle is considered as a fault-tolerant and dependable system. If one part of it breaks or does not work as expected, the human life is threatened. Since the intra-vehicular communication is the key element of the car, it is subject to safety certification. Safety certification is a process, where the manufacturer proves that his safety-critical systems such as autonomous driving are working correctly to a high degree of assurance. If they are not able to demonstrate the correct behavior of the central communication bus, then the whole certification process breaks down.

Traditionally, event-triggered communication protocols such as CAN (\textit{Controller Area Network}) are commonly used. In the event-triggered environment, the actions are performed \textit{on-demand}, i.e. triggered by some event. The communication is governed by scheduling policies that react to the observed situations during the run time execution. The capabilities of driver assistance systems are rapidly improving; hence the amount of data transferred through the network in a vehicle is growing. Traditional event-triggered protocols like CAN were not designed for a high data throughput; therefore their usage in modern cars is limited. Moreover, the response time analysis (i.e. the analysis of the behavior of the system) in real-life event-triggered communication systems including gateways and precedence relations is a very complex problem, therefore the safety certification of systems utilizing event-triggered environment is a difficult task.

\begin{figure*}[ht]
\centering
\includegraphics[width=0.8\textwidth]{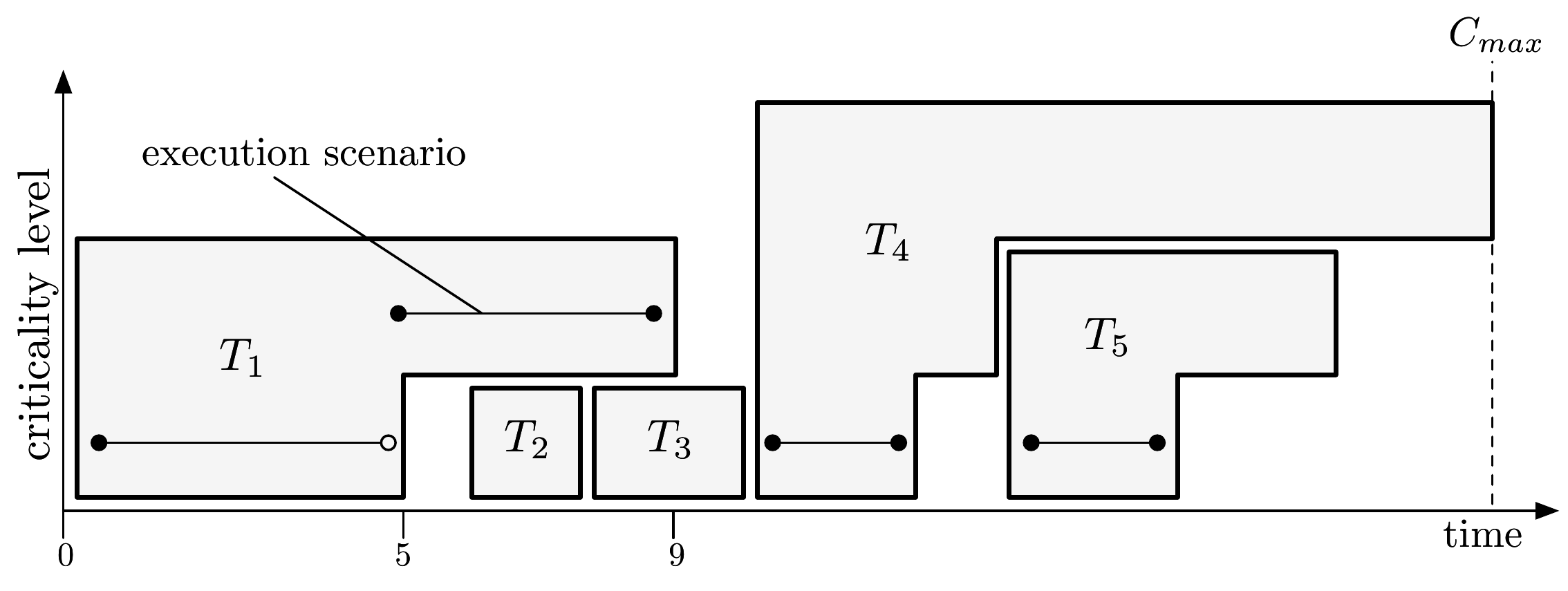} 
\caption{A feasible schedule of F-shaped tasks with a runtime execution scenario denoted by the black line.}
\label{fig:feasible}
\end{figure*}

Messages in time-triggered communication are transferred through the network at specific times prescribed by a static pre-computed static schedule. When constructing the schedule, a designer imposes a set of constraints the communication that is met in every feasible solution to the scheduling problem. Therefore, the certification of the system is be achieved via showing the feasibility of the produced communication schedule. This advantage of time-triggered environments leads to the design of new protocols that includes time-triggered communication for safety-critical systems. For example, FlexRay bus is used nowadays in the automotive industry (e.g. Porsche Panamera, Nissan Infinity Q50). The scheduling of FlexRay static segment can be solved very efficiently due to \cite{dvorak2016using}. Moreover, the modern time-triggered buses such as \textit{Time-Triggered Ethernet} \cite{kopetz2005time} offer a high throughput and determinism guarantees and enable applications like valet parking and autonomous driving.

One of the disadvantages of time-triggered protocols is their non-flexibility. For example, the static schedule does not take into consideration the message retransmission. The retransmission occurs when a highly critical message is not delivered e.g. due to electromagnetic noise. Typically in the complex systems functionalities with different criticality co-exists on a single bus. Suppose a typical case in a car with messages with three different levels of criticality:
\begin{itemize}
                \item messages with high criticality (criticality level 3) are used for safety-related functionalities (their failure may result in death or serious injury to people), such as steering and braking;
                \item messages with medium criticality (criticality level 2) are used for mission-related functionalities (their failure may prevent a goal-directed activity from being successfully completed), such as combustion engine control;
                \item messages with low criticality (criticality level 1) are used for infotainment functionalities, such as audio playback.    
\end{itemize}

Possible solution how to enable message retransmissions in static time-trigg\-ered schedules is to allocate more processing time for each message to account for possible retransmissions. If no retransmission occurs during an actual execution, then the resource is idle until the start time of the next message. However, since retransmissions are not that frequent, the average resource utilization is low.

We use a different strategy. We build static schedules that allow message retransmission to some degree. Extra time needed for retransmissions is compensated by skipping less critical messages (e.g. in our example, breaking versus audio playback). Even though this is a very elegant solution that achieves efficient resource utilization, the price we pay for it is that it modifies traditional scheduling problem into the challenging one -- the schedule has to assume alternative schedules based on the observed runtime scenario. There is an exponential number of possible runtime scenarios, and for each of them, the static schedule needs to be well-defined.

For this automotive application, the criticality of a message corresponds to its maximum number of possible (re)transmissions. See an example of a static schedule that accounts for retransmissions in Fig.~\ref{fig:feasible}. There $T_2$ and $T_3$ have low criticality, and no retransmissions are allowed. $T_1$ and $T_5$ correspond to messages with medium criticality; thus they can be re- transmitted once. The most critical message is $T_4$, that can be retransmitted twice. The retransmission of the messages causes a prolongation of the processing time that depicted in levels on the vertical axis. The top level of each message represents its WCET (\textit{the worst case execution time}), so this is the time that it takes to transmit the message under the most pessimistic conditions. The prolongations are compensated by skipping less critical messages. With this mechanism, the successful transmission of highly critical messages is guaranteed while in the average case runtime scenario the resource (i.e. communication bus) is efficiently utilized. Schedules with three criticality levels arise from the application in the automotive domain. The adaptation of IEC 61508 safety standard \cite{bell2006introduction},  the \textit{ASIL}, defines the application levels with a hazard assessment corresponding to three \textit{Safety Integrity Levels}.  

Scheduling of safety-critical messages on this time-triggered network can be modeled as a scheduling problem where tasks having a set of different processing times represent messages, and the resource is a communication channel in the network. A solution of the scheduling problem is given by a schedule that switches to the higher criticality level when a prolongation of a task occurs. After its successful completion, it matches-up with the original schedule. The trade-off between the safe and efficient schedules is achieved by skipping less critical messages when the prolongation of a more critical one takes place.

The problem of mixed-criticality message retransmission in time-triggered environments leads to an interesting combinatorial problem, where we are given a set of shapes that are aligned on the left side with the right side that is jagged  (see an example in Fig.~\ref{fig:feasible}). The goal is to pack these shapes as tight as possible so that they do not overlap.

\subsection{Contribution and Paper Outline}
In this paper, we solve the scheduling problem of message retransmission in time-triggered environments. The objective is to find a static schedule that accounts for unforeseen message retransmissions while minimizing the length occupied by time-triggered communication. 
The uncertainty about the processing time is modeled using an abstraction that considers F-shaped tasks. We show the relation be- tween F-shaped tasks and the underlying probability distribution functions. Furthermore, we show a new complexity result that establishes the membership of the considered problem into $\mathcal{APX}$ complexity class, and we provide an approximation algorithm. We study the characterization of the set of optimal solutions for the problem with two criticality levels. Finally, we propose efficient exact algorithms for problems with two and three criticality levels, which solve instances with up to 200 tasks, beating the best-known method by a large margin.

The rest of the paper is organized as follows. In Sec.~\ref{sec:relwork} we survey the related work. In Sec.~\ref{sec:mcproblem} we show the relation between F-shaped tasks and discretization of cumulative probability distribution functions. In Sec.~\ref{sec:genprop} we prove approximability of the problem. In Sec.~\ref{sec:two} and \ref{sec:three} we show properties of the problem with two and three criticality levels and we propose efficient exact algorithms. Finally, in Sec.~\ref{sec:results} we present computational results on sythetic data demonstrating the efficiency of the proposed approach.

\section{RELATED WORK} \label{sec:relwork}



The exhaustive survey on mixed-criticality in real-time systems is presented by \cite{burns2013mixed}. This research is traditionally concentrated around event-triggered approach to scheduling. In the seminal paper \cite{vestal2007preemptive} Vestal proposed a method that assumes different WCETs (\textit{the worst case execution time}) obtained for discrete levels of assurance. Apart from this proposition, the paper presents modified preemptive fixed priority schedulability analysis algorithms. However, the preemptive model is not suitable for communication protocols, and it significantly changes the scheduling problem. \cite{baruah2010towards} formulated the basic model of mixed-criticality systems. They study MC schedulability problem with two criticality levels under special restrictive cases in the event-triggered environment. \cite{theis2013schedule} argued that mixed-criticality shall be pursued in time-triggered systems. Baruah's approach \cite{baruah2011certification} in the time-triggered environment assumed preemptive tasks with up to two criticality levels. It makes it unsuitable for communication protocols since the preemption would be costly. \cite{hanzalek2016mc} proposed the problem of non-preemptive mixed-criticality match-up scheduling motivated by scheduling messages on a highly used communication channel. They showed how a schedule with F-shaped tasks can be used to deal with a task disruption by skipping less critical tasks. They provide the relative order MILP model for $1|r_j,\tilde{d}_j,mc=\mathcal{L},mu|\cmax$ scheduling problem, but it can deal with instances with only about 20 messages.

The concept of match-up scheduling was introduced by \cite{bean1991matchup}. In a case of a disruption, the goal is to construct a new schedule that matches the original one at some  point in the future. This concept is mostly studied in the context of manufacturing problems \cite{qi2006disruption}. 


Taking broader perspective, the problem can be viewed as a case of robust and stochastic optimization due to uncertainty about transmission times while satisfying safety requirements. \cite{bertsimas2011theory} surveys robust versions of various optimization problems, but rather continuous than discrete ones. The field of stochastic optimization is reviewed by \cite{sahinidis2004optimization}. They state that integer variables introduced to stochastic programming complicate its solution, yielding suboptimal results even for small-sized problems.


As in our problem, some of the less critical messages are allowed to be skipped, the problem is related to the scheduling with a job rejection. \cite{shabtay2013survey} reviews offline scheduling with a job rejection. These approaches consider two criteria, a measure associated with schedule quality and the cost incurred by rejected jobs. The solution to this problem is a set of accepted jobs and a set of rejected jobs. However, rejected jobs cannot be executed in any execution scenario; thus this model is not suitable for communication protocols mentioned in our motivation. Our problem also embeds scheduling with setup times. \cite{allahverdi2015third} shows that problems with \textit{sequence-dependent} (i.e. where a setup time is given for a pair of consecutive tasks) setup times are mostly studied. However, in our problem the feasible start time of a task depends on a permutation of \textit{all} preceding tasks, not just the immediate predecessor; therefore it represents a more general problem.


To the best of our knowledge, the problem of offline non-preemptive mixed-criticality match-up scheduling was addressed by \cite{hanzalek2016mc} only, but it lacks an efficient solution method which is suggested in this paper.



\section{NON-PREEMPTIVE MIXED-CRITICALITY SCHEDULING} \label{sec:mcproblem}

\begin{figure*}[ht]
\centering
\begin{subfigure}[t]{0.42\textwidth}
\centering
    \includegraphics[width=\textwidth]{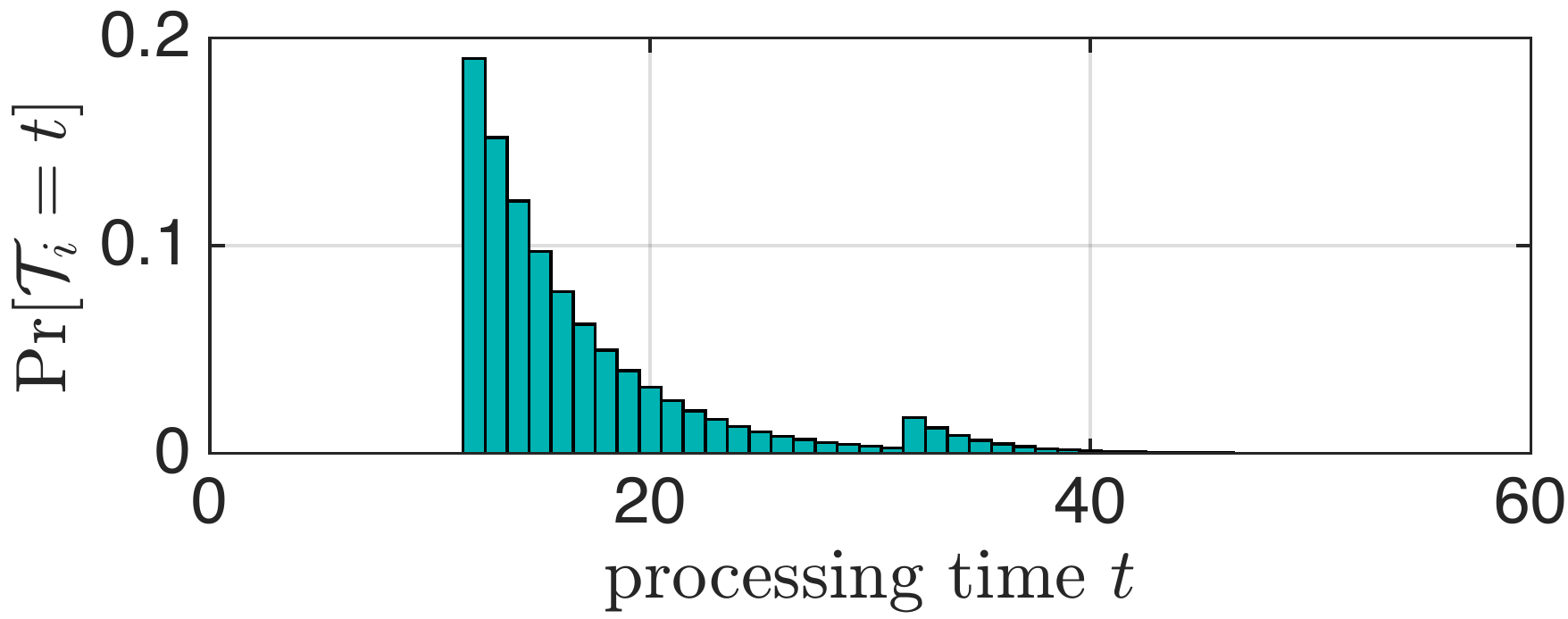} 
 \caption{A discrete probability distribution over a set of \mbox{processing} times.}
 \label{fig:pst_pdf}
\end{subfigure}
\quad
\quad
\quad
\quad
\begin{subfigure}[t]{0.42\textwidth}
\centering
   \includegraphics[width=\textwidth]{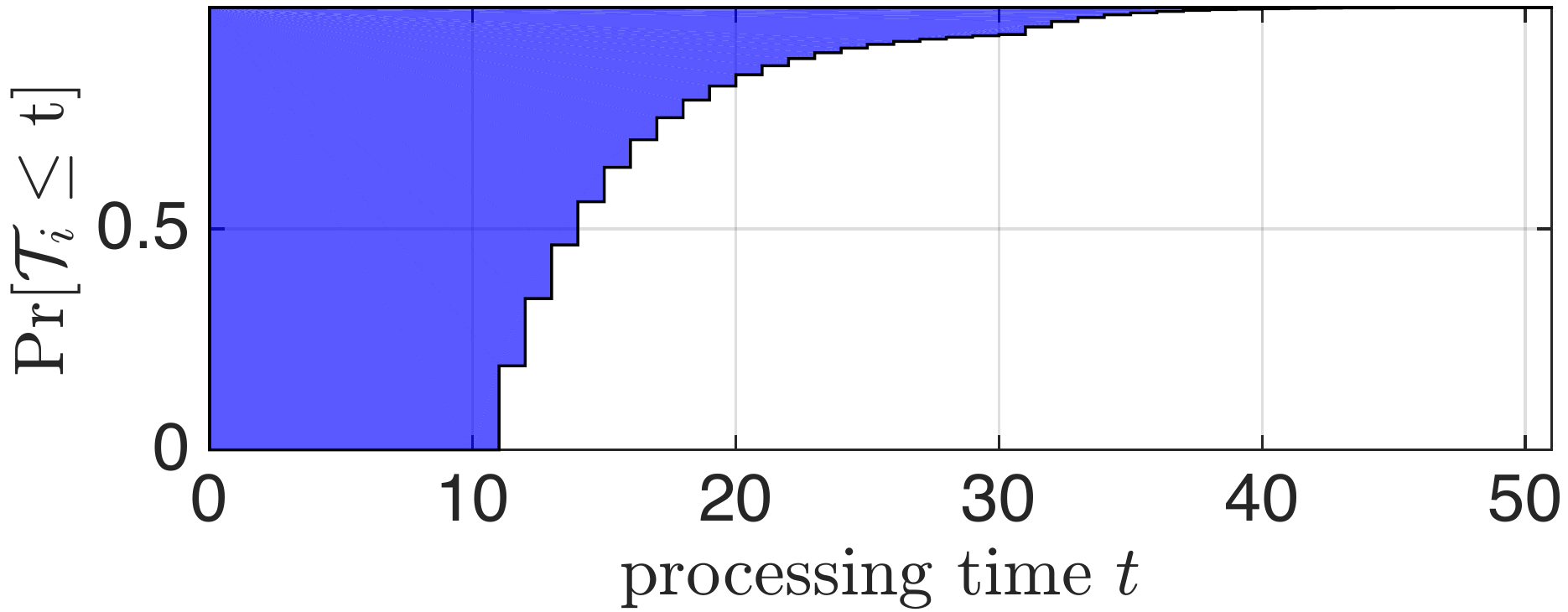}
\caption{Corresponding cumulative distribution function.}
\label{fig:pst_cdf}
\end{subfigure}

\vspace{1em}

\begin{subfigure}[b]{0.42\textwidth}
  \centering
  \includegraphics[width=1\linewidth]{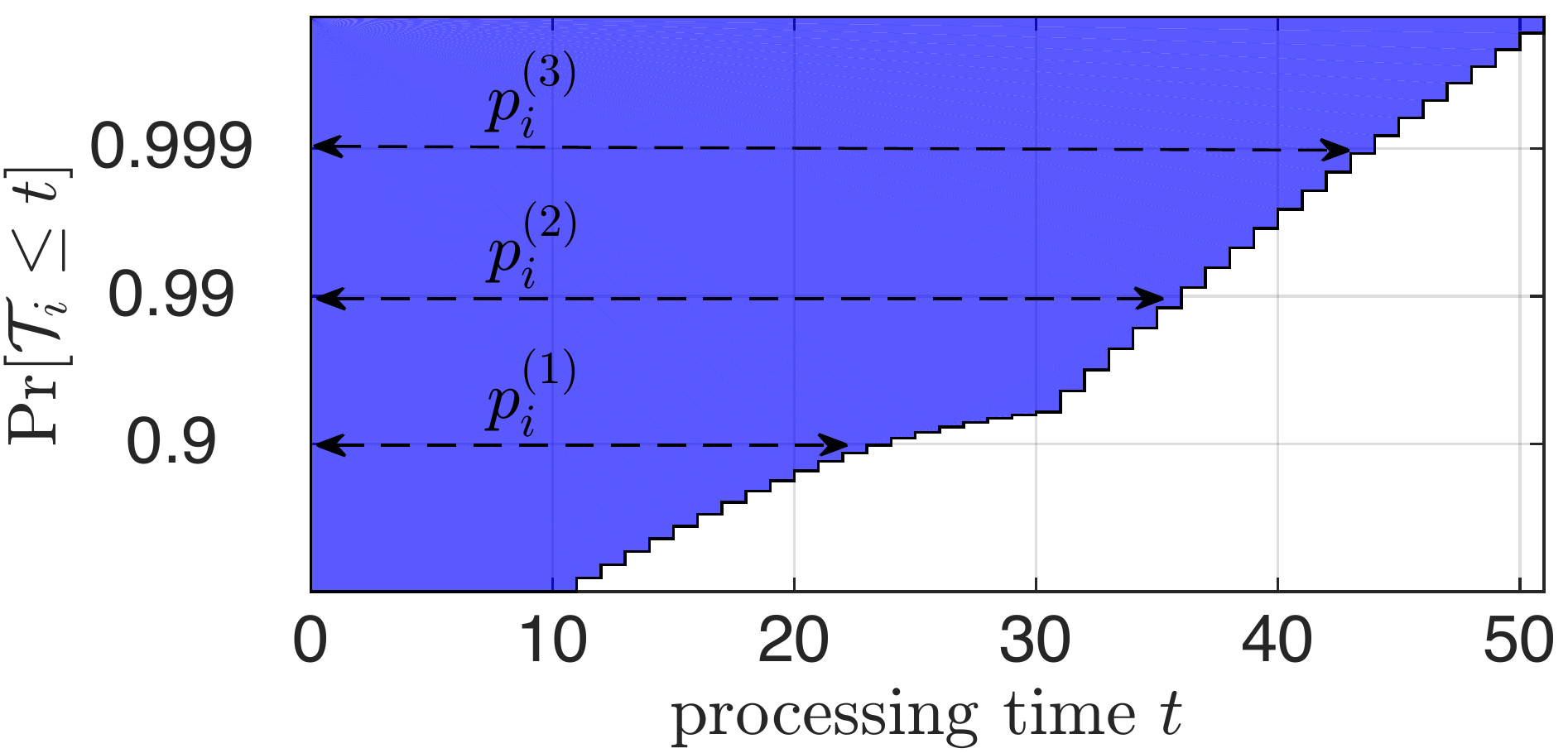}
  \caption{Discretization of a cumulative distribution function}
  \label{fig:fshape-sampling}
\end{subfigure}%
\quad
\quad
\quad
\quad
\begin{subfigure}[b]{0.42\textwidth}
  \centering
  \includegraphics[width=1\linewidth]{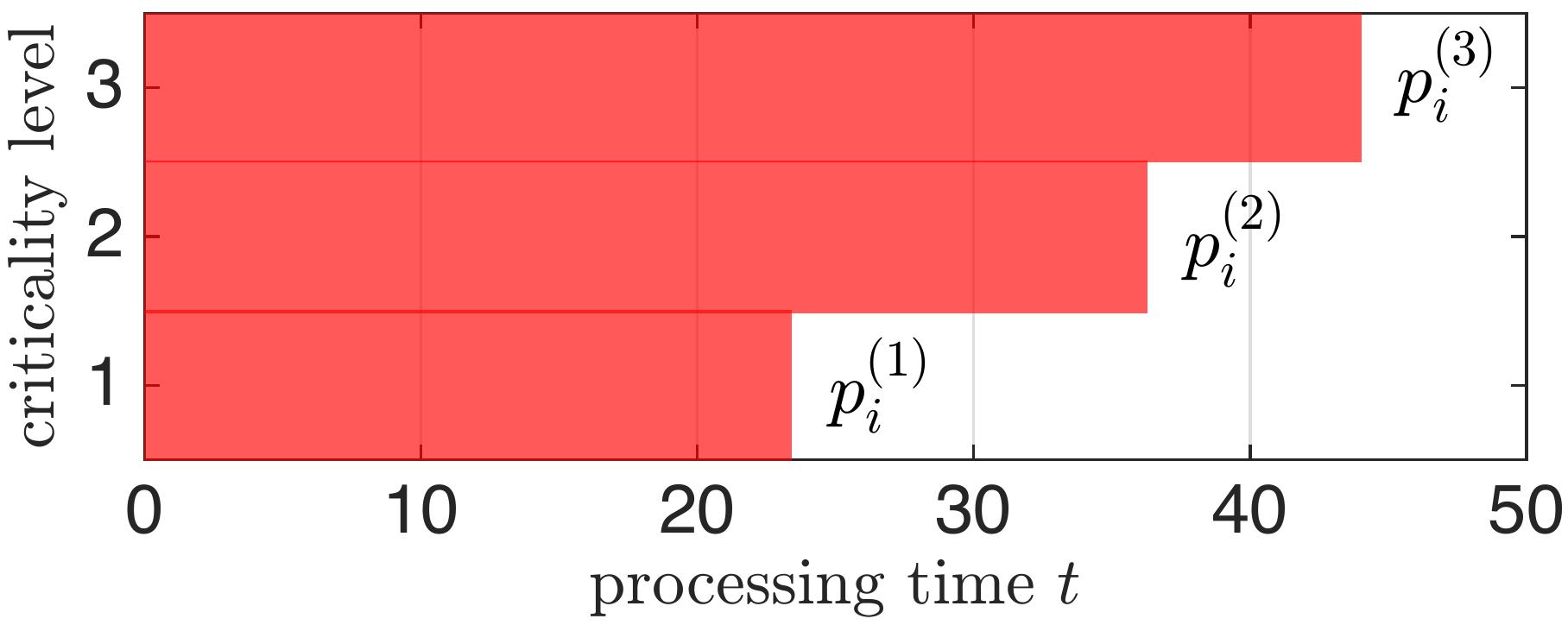}
  \caption{Resulting F-shaped task with $\mathcal{X}_i=3$}
  \label{fig:fshape}
\end{subfigure}
\caption{Discretized cumulative distribution function forms an F-shaped task}
\label{fig:disretization}
\end{figure*}

We assume that a set of communication messages is given to be scheduled on a single communication bus segment. For each message $T_i$, the criticality $\mathcal{X}_i \in \mathbb{N}$ is specified. It denotes the number of allowed transmissions. Each message (task) is specified by its \textit{criticality levels}. For each \emph{criticality level} $\ell \in \{1, \ldots, \mathcal{X}_i\}$, we define the associated processing time with this level. See an example in Fig.~\ref{fig:feasible}. Here, $T_1$ has criticality $\mathcal{X}_1 = 2$; therefore it can be retransmitted once. The processing time at the first level is given by its BCET (\textit{the best case execution time}) while the processing time at the second level is its WCET (\textit{the worst case execution time}). During the run time execution, exactly one processing time of the message is realized; however, it is not known in advance which it will be.

We can view processing time prolongations as a retransmission of the whole message content. However, this mixed-criticality scheduling model is useful also for scheduling of computational tasks, where the exact computational time is not known in advance, but only a probability distribution is known.  Let us consider the processing time of task $T_i$ to be a random variable $\mathcal{T}_i$. Let us assume an arbitrary probability distribution over a discrete set of processing times from $\mathbb{N}$ for a particular task stating $\text{Pr}[\mathcal{T}_i=t]$. The same information given by the probability distribution is captured by the CDF (\textit{cumulative distribution function}) $\mathcal{F}_i$ giving the probability that processing time $\mathcal{T}_i$ is at most $t$. See Figs.~\ref{fig:pst_pdf} and \ref{fig:pst_cdf} for such an example. Corresponding processing times $p_i^{(\ell)}$ for each criticality level $\ell$ are taken from $\mathcal{F}_i$ as $p^{(\ell)}_i = \mathcal{F}^{-1}_i(c_\ell)$, where $ \mathcal{F}^{-1}_i$ is the quantile function. The criticality of a task is a user-defined parameter. For example, if we identify criticality levels with a safety standard IEC 61508 SIL (\textit{Safety Integrity Levels}) \cite{bell2006introduction}, then the task criticality $\mathcal{X}_i$ is given by the SIL of the functionality carried out by the content and $c_\ell$ is defined as $1-$\emph{probability of failure} defined by SIL $\ell$. 

Processing times obtained according to criticality levels then form a single task like the one depicted in Fig.~\ref{fig:fshape}. Since CDFs are non-decreasing functions, a set of processing times $p^{(\ell)}_i$ yields shapes like the F letter rather than ordinary rectangles, hence the name F-shape. See an example in Fig~\ref{fig:disretization}. There we see discretization for a task with criticality three at corresponding levels $1$, $2$ and $3$ with the vertical axis on the logarithmic scale.

\subsection{Execution Policy}
The solution to the scheduling is a feasible static schedule of the given set of F-shaped tasks. Consider a particular example of the schedule with tasks having up to three levels of criticality that is shown in Fig.~\ref{fig:feasible}. A feasible schedule with F-shaped tasks describes alternative schedules for any realization of the processing time of messages. Observed prolongations of more critical messages are compensated by skipping execution of less critical messages.

The black line denotes a scenario, where $T_1$ was disrupted once. The actual processing time of $T_1$ was $9$ instead of $5$ due to a disturbance. When the disruption occurred, execution switched to the next higher criticality level. There, by the assumption, the execution was successful with a probability given by the $\ell = 2$ criticality level. After upon $T_1$ finished, the execution matched-up back with the lowest criticality level. In general, if a task $T_i$ is prolonged to level $\ell$, then all tasks $T_j$ for which $s_i + p^{(1)}_i \leq s_j < s_i + p^{(\ell)}_i$ are not executed. Therefore, in this execution scenario, after $T_1$ finished, $T_4$ was up next. Moreover, if we unify the F-shape from Fig.~\ref{fig:fshape} with task $T_4$ in Fig.~\ref{fig:feasible}, then we can say that $T_5$ will be executed with very high probability of $0.99$, but in rare cases, it won't be executed since $T_4$ is more critical and needs more time to complete.
\section{PROBLEM STATEMENT}
We assume a set of non-preemptive F-shaped tasks $I_\mathcal{MC} = \{T_1, \ldots, T_n\}$ to be processed on a single machine. We define an F-shaped task and its criticality as follows:
\begin{mydef}[F-shape]
The F-shape $T_i$ is a pair $(\mathcal{X}_i, \m{P}_i)$ where $\mathcal{X}_i\in\{1, \ldots, \mathcal{L}\}$, $\mathcal{L}\in\mathbb{N}$ is the task criticality and $\m{P}_i\in\mathbb{N}^{\mathcal{X}_i}$, $\m{P}_i = (p_i^{(1)}, p_i^{(2)}, \ldots, p_i^{(\mathcal{X}_i)})$ is the vector of processing times such that 
$$p_i^{(1)} < p_i^{(2)} < \ldots < p_i^{(\mathcal{X}_i)}.$$
\end{mydef}
The F-shape is an abstraction for non-preemptive tasks with multiple different processing times. See for example $T_4$ in Fig.~\ref{fig:feasible}. It is F-shaped task with criticality $\mathcal{X}_4 = 3$; therefore it has 3 different processing times. Having a set $I_\mathcal{MC}$ of F-shaped tasks, we define the feasible schedule as follows:
\begin{mydef}[Feasible Schedule]
By the schedule for a set of F-shaped tasks $I_\mathcal{MC} = \{T_1, T_2, \ldots, T_n\}$ we refer to an assignment $(s_1, s_2, \ldots, s_n) \in\mathbb{N}^n$. We say that schedule $(s_1, s_2, \ldots, s_n)$ for $I_\mathcal{MC}$ is feasible if and only if $\,\forall i,j \in \{1, \ldots, n\}, i\neq j:$ 
$$(s_i + p^{(\min\{ \mathcal{X}_i, \mathcal{X}_j\})}_i \leq s_j) \,\vee\, (s_j + p^{(\min \{ \mathcal{X}_i, \mathcal{X}_j\})}_j \leq s_i).$$
\end{mydef}
Feasibility of a schedule with F-shaped tasks requires that tasks do not overlap on any criticality level. For example in Fig.~\ref{fig:feasible}, since $T_5$ follows after $T_4$, it cannot start earlier than $s_4 + p^{(2)}_4$, since $\min\{\mathcal{X}_4, \mathcal{X}_5\} = 2$  is the highest common criticality level of $T_4$ and $T_5$.

We deal with the problem of finding a feasible schedule for a set of F-shaped tasks with criticality at most $\mathcal{L}$ such that the makespan (i.e. $\max s_i + p_i^{(\mathcal{X}_i)}$) is minimized. In the three-field Graham-Blazewicz notation it is denoted as $1|mc=\mathcal{L},mu|\cmax$, where $mc=\mathcal{L}$ stands for the mixed-criticality aspect of tasks of maximal criticality $\mathcal{L}$ and $mu$ stands for the match-up. This problem is known to be $\mathcal{NP}$-hard in the strong sense even for $mc=2$ (two criticality levels) as shown by reduction from 3-Partition Problem in \cite{hanzalek2016mc}.

\section{GENERAL PROPERTIES} \label{sec:genprop}


Since the problem $1|mc=2,mu|\cmax$ is strongly $\mathcal{NP}$-hard, it does not admit FPTAS unless $\mathcal{P}=\mathcal{NP}$. However, we show that the problem is polynomial-time approximable within a constant multiplicative factor.

\begin{prop}[Approximability]\label{prop:approx}
For any given fixed $\mathcal{L}$, the problem $1|mc=\mathcal{L},mu|\cmax$  is contained in $\mathcal{APX}$ complexity class.
\end{prop}
\begin{proof}
Suppose the algorithm \emph{LCF (Least Criticality First)} that takes an input instance $I_\mathcal{MC}$ and schedules tasks in a non-decreasing sequence by their criticalities without waiting. Then the makespan of resulting schedule is 
\begin{equation*}
\begin{split}
LCF(I_\mathcal{MC}) = \sum_{ i | \mathcal{X}_i = 1} p^{(1)}_i + \sum_{ i | \mathcal{X}_i = 2} p^{(2)}_i + \ldots + \sum_{ i | \mathcal{X}_i = \mathcal{L}} p^{(\mathcal{L})}_i
\end{split}
\end{equation*}
A sum of processing times on a given criticality level over a set of tasks is a lower bound on the makespan. Therefore we have
\begin{equation*}
\begin{split}
\max\{\sum_{ i | \mathcal{X}_i \geq 1} p^{(1)}_i, \sum_{ i | \mathcal{X}_i \geq 2} p^{(2)}_i, \ldots, \sum_{ i | \mathcal{X}_i \geq \mathcal{L}} p^{(\mathcal{L})}_i\}\leq \\ \leq OPT(I_\mathcal{MC}) 
\leq LCF(I_\mathcal{MC}) \leq \mathcal{L}\cdot OPT(I_\mathcal{MC}).
\end{split}
\end{equation*}
where $OPT(I_\mathcal{MC})$ denotes the optimal makespan of $I_\mathcal{MC}$ problem instance.
\end{proof}

In fact, this result shows more than that there exists a polynomial-time algorithm producing schedules with a constant bounded quality. For example, for the problem with $\mathcal{L}=2$ criticality levels, actually any left-shifted schedule will be at most twice as worse as the optimal makespan since $LCF$ actually produces the worst ordering of tasks in terms of the makespan. 

In the following sections, we present exact algorithms for the problem with $2$ and $3$ criticality levels. Due to the $\cmax$ criterion, it can be shown that the search for an optimal solution can be reduced to finding a permutation of tasks. Therefore, any optimal schedule is given by a permutation of tasks $\pi$. Hence we denote the makespan of the left-shifted schedule of permutation $\pi$ by $\cmax(\pi)$. In Section~\ref{sec:two} we give a characterization of the set of optimal permutations for problem $1|mc=2,mu|\cmax$ and we introduce a MILP model utilizing it. In Section~\ref{sec:three}, we introduce an operator acting on F-shapes, and we show how the optimal solutions for problems with two and three criticality levels are related.

\section{TWO CRITICALITY LEVELS} \label{sec:two}
We showed that optimal solutions to $1|mc=\mathcal{L},mu|\cmax$ are given by a permutation $\pi$ of tasks. For the problem with two criticality levels, the optimal permutations can be characterized more precisely. Let us refer to tasks with criticality $\mathcal{X}_i = 2$ as \textsc{Hi}-tasks and tasks with criticality $\mathcal{X}_j=1$ as \textsc{Lo}-tasks. The key structure of the optimal permutations are \textit{covering blocks}:

\begin{mydef}[Covering Block]
For any given feasible schedule $(s_1, \ldots, s_n)$, a \textsc{Hi}-task $T_i$ and a \textsc{Lo}-task $T_j$ we say that $T_j$ is covered by $T_i$, denoted as $T_i \in \text{cov}(T_j)$, if and only if $s_i +p^{(1)}_i \leq s_j < s_i + p^{(2)}_i$. The covering block $B_i$ is then the \textsc{Hi}-task $T_i$ and the set of all \textsc{Lo}-tasks covered by $T_i$.
\end{mydef}
See an example in Fig.~\ref{fig:feasible}. There $T_{1}$ is covering $T_2$ and $T_3$. All these tasks form a covering block. Although the definition of covering block given above is meant for the problem with two criticality levels, the notion of \textit{covering} can be generalized for more criticality levels. We assign a \textit{length} to each covering block. The length is given as the maximum between the processing time $p^{(2)}_i$ of the \textsc{Hi}-task $T_i$ and the sum of processing times of tasks covered by $T_i$ plus the processing time of $T_i$ at the first level $p^{(1)}_i$. 
\begin{prop}[Covering Block Length] \label{prop:lo-comm}
Given the covering block $B_i$, its length defined as  
$$\max\{p^{(1)}_i + \sum_{T_j|\, T_i\in \text{cov}(T_j)} p^{(1)}_j,  \quad p^{(2)}_i\}$$
 is invariant with respect to  the ordering of \textsc{Lo}-tasks $T_j$ for which $T_i \in \text{cov}(T_j)$.
\end{prop}
Clearly, the ordering of \textsc{Lo}-tasks $T_j$ for which $T_i \in \text{cov}(T_j)$ does not affect the block length since all \textsc{Lo}-tasks are running without waiting. Furthermore, we say that task $T_j$ is \textit{fully covered} by the block $B_i$, if $B_i = p^{(2)}_i$ and $T_i \in \text{cov}(T_j)$. If exists a task covered by the block $B_i$ that is not fully covered, then we say that $B_i$ is \textit{saturated}. The makespan $C_{max}$ of the schedule is given by a permutation of covering blocks. However, actually any permutation of covering blocks contributes to the makespan by the same amount; hence it is not subject to optimization.
\begin{prop}[Interchangebility] \label{prop:blocks-comm}
For every instance of the problem $1|mc=2,mu|\cmax$ there exists an optimal solution that is given by an arbitrary permutation of covering blocks.
\end{prop}
A characterization of optimal solutions for $1|mc=2,mu|\cmax$ directly follows from Proposition \ref{prop:lo-comm} and \ref{prop:blocks-comm}:
\begin{coroll}\label{cor:mc2}
The optimal solution for $1|mc=2,mu|\cmax$ is given by an assignment of \textsc{Lo}-tasks to \textsc{Hi}-tasks.
\end{coroll}

\subsection{Covering MILP Model for $1|mc=2,mu|\cmax$} \label{milp:mc2}
The following MILP model relies on Corollary~\ref{cor:mc2}. The model assigns \textsc{Lo}-tasks to the \textsc{Hi}-tasks in order to form covering blocks such that the sum of their lengths is the minimal one. The decision variable $x_{ij}$ indicates whether the \textsc{Lo}-task $T_j$ is covered by the \textsc{Hi}-task $T_i$; therefore if $T_i \in \text{cov}(T_j)$, then $x_{ij}=1$. The makespan is then given by the sum of lengths of covering blocks and the sum of processing times of all \textsc{Lo}-tasks that are not covered.
\begin{align}
&\min \sum_{ i|\mathcal{X}_i=2} B_i + \sum_{ j|\mathcal{X}_j=1}p^{(1)}_j(1-\sum_{ i|\mathcal{X}_i=2}x_{ij}) \label{eq:makespan}\\
\text{s.t.} \notag\\
&B_i \geq p^{(1)}_i + \sum_{ j|\mathcal{X}_j=1} p^{(1)}_j x_{ij}\quad \forall i\in I_\mathcal{MC}|_{\mathcal{X}_i=2}\\
&B_i \geq p^{(2)}_i \quad \forall i\in  I_\mathcal{MC}|_{\mathcal{X}_i=2}\\
&\sum_{ i|\mathcal{X}_i=2} x_{ij}\leq 1 \quad \forall j\in I_\mathcal{MC}|_{\mathcal{X}_j=1}\\
\text{where} \notag \\
&B_i\in \mathbb{Z}^+_0 \quad \forall i\in I_\mathcal{MC}|_{\mathcal{X}_i=2} \notag \\ 
&x_{ij} \in \{0,1\} \quad \forall i \in I_\mathcal{MC}|_{\mathcal{X}_i=2}, \forall j\in I_\mathcal{MC}|_{\mathcal{X}_j=1} \notag
\end{align}

The main advantage of this model over the model proposed by \cite{hanzalek2016mc} is that it has much stronger linear relaxation. In Section~\ref{sec:results} we show that it can solve instances with about the order of magnitude more tasks. 
\section{THREE CRITICALITY LEVELS} \label{sec:three}
Although two criticality levels are often sufficient for safety-critical application and this case is frequently studied in the field of mixed-critical systems \cite{burns2013mixed}, sometimes the application naturally contains three or more criticality levels. We capture the direct relation between problems with different maximum criticality levels by introducing a transformation given bellow. It is based on the observation that omitting some criticality levels provides an instance of the problem with less criticality level while maintaining a lower bound property. Furthermore, we introduce the \textit{Bottom-up} algorithm that uses this observation. The algorithm is used then together with Covering MILP model for three criticality levels ($\mathcal{L}=3$) shown in Section~\ref{milp:mc3} to form an efficient solution method. 

The transformation is defined as $h^\pm$ restrictions:
\begin{mydef}[$h^\pm$ restrictions]
Given the mixed-criticality instance $I_\mathcal{MC}$ and a positive integer $h\in\mathbb{N}$, let $I^{h^-}_\mathcal{MC}$ and $I^{h^+}_\mathcal{MC}$ be sets defined as 
\begin{align*}
I^{h^-}_\mathcal{MC} = &\{(\min\{h, \mathcal{X}_i\},\, (p^{(1)}_i, \ldots, p^{(\min\{ h, \mathcal{X}_i\})}_i)) \; |\, \forall i\in I_\mathcal{MC}\} \notag\\
I^{h^+}_\mathcal{MC} = &\{(\mathcal{X}_i-h+1,\,  (p^{(h)}_i, \ldots, p^{(\mathcal{X}_i)}_i)) \; |\, \forall i\in I_\mathcal{MC}: \mathcal{X}_i \geq h \} \notag
\end{align*}
We refer to $I^{h^-}_\mathcal{MC}$ ($I^{h^+}_\mathcal{MC}$)  as $h^-$ ($h^+$) restriction of the instance  $I_\mathcal{MC}$.
\end{mydef}
The $h^-$ restriction takes an F-shape and cuts off all criticality levels above level $h$. Similarly, given the set of F-shaped tasks, $h^+$ restriction drops all tasks with criticality below $h$, and for the rest, it cuts off criticality levels less than $h$. Restricting an $I_\mathcal{MC}$ instance yields to a mixed-criticality instance since omitting some of the criticality levels for an F-shape gives us an F-shape. The application of the restriction can be viewed as a relaxation the problem.

\begin{prop}[Two Lower Bounds on the Makespan] \label{prop:lb}
For the problem \\$1|mc=3, mu|\cmax$ expressions $lb^-$, $lb^+$ defined as
$$lb^\pm = \min_{\pi\in \Pi(I^{2^\pm}_\mathcal{MC})} \cmax(\pi)$$
are lower bounds on the makespan, where $\Pi(I^{2^+}_\mathcal{MC})$ and $\Pi(I^{2^-}_\mathcal{MC})$ denote the set of all permutations of elements $I^{2^+}_\mathcal{MC}$, $I^{2^-}_\mathcal{MC}$ respectively.
\end{prop}
\begin{proof}
The $lb^-$ is a lower bound on the makespan of $1|mc=3,mu|\cmax$ since it relaxes on the overlapping condition at the third criticality level. Similarly, $lb^+$ is a lower bound on the makespan since it relaxes on the overlapping condition at the first criticality level. 
\end{proof}

\subsection{Bottom-up Algorithm} \label{sec:botup}
We introduce a heuristic algorithm for the problem $1|mc=3,mu|\cmax$. Let us refer to tasks with $\mathcal{X}_i=3$ (i.e. criticality 3) as to \textsc{Great}-tasks. The \textit{Bottom-up} algorithm is based on the idea of constructing the schedule in two stages. In the first stage, the relaxed problem is solved up to the optimality, which minimizes a lower bound on the optimal makespan of the original problem. The second stage takes the relaxed solution and constructs a locally optimal solution for the original problem. 

The first stage of the algorithm solves $2^-$ restriction of the given problem instance; hence it is an instance of $1|mc=2,mu|\cmax$ problem that can be solved with the model described in Section~\ref{sec:two}. It assigns \textsc{Lo}-tasks to \textsc{Hi}-tasks and  \textsc{Great}-tasks; therefore it forms covering blocks. In the second stage, the algorithm defines a new problem instance $I^\prime_\mathcal{MC}$ of the problem $1|mc=2,mu|\cmax$. The instance is constructed as follows. It contains \textsc{Lo}-tasks with processing time equal to the length of covering blocks from the stage one. \textsc{Lo}-tasks that are not part of any covering block are assigned to an arbitrary covering block. The assignment of \textsc{Lo}-tasks to $2^-$ restricted \textsc{Great}-tasks from the first stage defines \textsc{Hi}-tasks in the new instance $I^\prime_\mathcal{MC}$. Then, the $I^\prime_\mathcal{MC}$ instance is solved once again as an instance of $1|mc=2,mu|\cmax$ problem. See the complete description of the \textit{Bottom-up} algorithm in Alg.~\ref{alg:botup}. 

In general, the \textit{Bottom-up} algorithm produces suboptimal solutions even though they are provably bounded by a factor of $3$ from the optimal solution, as stated by Proposition~\ref{prop:approx}. However, there are cases when we can verify if the produced schedule is optimal. This is achieved by the concept of \textit{critical paths} that captures the cause of achieved makespan.


\begin{algorithm}
\caption{Bottom-up} \label{alg:botup}
\begin{algorithmic}[1]
\State $\pi \gets \text{solve } I^{2^-}_\mathcal{MC}$ restriction by Covering MILP~\ref{milp:mc2}
\State $I^\prime_\mathcal{MC} \gets \emptyset$
\ForAll{covering block $B_i$ in the left-shifted solution $\pi$}  
	\If{$\mathcal{X}_i = 2$ in $I_\mathcal{MC}$}
		\State $\m{P}_i\gets (B_i)$
		\State $I^\prime_\mathcal{MC}\gets I^\prime_\mathcal{MC}\cup\{(1, \m{P}_i)\}$
	\ElsIf{$B_i < p^{(3)}_i$}
		\State $\m{P}_i\gets (B_i, \, p^{(3)}_i)$
		\State $I^\prime_\mathcal{MC}\gets I^\prime_\mathcal{MC}\cup\{(2, \m{P}_i)\}$
	\Else
		\State \Comment{block $B_i$ is \textit{saturated}, it contributes by a constant term to the makespan of $I^\prime_\mathcal{MC}$}
	\EndIf 
\EndFor
\State $\pi\gets$ solve $I^\prime_\mathcal{MC}$ by Covering MILP~\ref{milp:mc2}
\end{algorithmic}
\end{algorithm}

\begin{mydef}[Critical Path]
Given the left-shifted schedule $(s_1, \ldots, s_n)$ of the permutation $\pi$, the critical path is $\mathcal{CP}\subseteq \{1, \ldots, |\tilde{\pi}|\}\times \{1, \ldots, \mathcal{L}\}$ for some $\tilde{\pi}\subseteq\pi$ such that $\forall (i, \ell)\in \mathcal{CP}, i < |\tilde{\pi}|: \, s_{\tilde{\pi}(i)} + p^{(\ell)}_{\tilde{\pi}(i)} = s_{\tilde{\pi}(i+1)}$ where $\sum_{(i, \ell)\in\mathcal{CP}} p^{(\ell)}_{\tilde{\pi}(i)} = \cmax(\pi) = \cmax(\tilde{\pi})$.
\end{mydef} 

Essentially, for any given left-shifted schedule, the critical path is a subset of tasks and their  criticality levels such that $\forall (i, \ell) \in \mathcal{CP}$ holds that if the processing time $p^{(\ell)}_{\tilde{\pi}(i)}$ is increased by some $\epsilon > 0$, then the makespan of the same schedule is also increased by $\epsilon$.

\begin{prop}[Sufficient Optimality Conditions] \label{prop:opt}
If one of following conditions holds, then the schedule produced by the Bottom-up algorithm is optimal for problem $1|mc=3,mu|\cmax$.
\begin{enumerate}
\item There exists a critical path going through the first and the second levels only.
\item Every \textsc{Lo}-task is fully covered by the second criticality level.
\end{enumerate}
\end{prop}

When none of the optimality conditions is satisfied, e.g. a critical path is coming through every criticality level, we fallback to the MILP model~\ref{milp:mc3} for the problem $1|mc=3,mu|\cmax$ in order to find an optimal solution or for the proof that the current solution is the optimal one. The solver is supplied with the initial solution and a lower bound obtained by the  \textit{Bottom-up} algorithm. 

\subsection{Covering MILP Model for $1|mc=3,mu|\cmax$} \label{milp:mc3}
The Covering MILP model for three criticality levels uses a similar idea as the model for $1|mc=2,mu|\cmax$. It assigns \textsc{Lo}-tasks to covering blocks and covering blocks to the \textsc{Great}-tasks. The model utilizes the idea that optimal solutions are made of blocks (in this case formed by \textsc{Great}-tasks that cover less critical tasks) whose order is interchangeable within a solution. It assigns \textsc{Lo}-tasks to the \textsc{Hi}-tasks and to $2^-$ restriction of \textsc{Great}-tasks to form covering blocks. Blocks are assigned to the \textsc{Great}-tasks in order to create a solution. The big $M$ constant is as large as the number of \textsc{Lo}-tasks contained in the problem instance.
\begin{align}
&\min \sum_{i|\mathcal{X}_i=3}p_i + \sum_{j|\mathcal{X}_j=2}P_{j,\emptyset} +\sum_{k|\mathcal{X}_k = 1} p^{\text{(1)}}_k x_{\emptyset,\emptyset,k}\\
\text{s.t.} \notag\\
&p_i \geq  p^{\text{(3)}}_i \quad \forall i\in \hi \\
&My_{i,j} \geq \sum_{k|\mathcal{X}_k = 1} x_{i,j,k} \quad \forall i\in \hi\cup\emptyset, \forall j\in \midlvl\\
&P_{j,i} \geq p^{\text{(2)}}_jy_{i,j} \forall i\in\hi\cup\emptyset, \forall j\in \midlvl\\
&P_{j,i} \geq p^{\text{(1)}}_jy_{i,j} +\sum_{k|\mathcal{X}_k = 1}p^{\text{(1)}}_k x_{i,j,k} \notag \\ 
&\forall i\in \hi\cup\emptyset, \forall j\in \midlvl \\
&p_i \geq p^{\text{(2)}}_i + \sum_{j|\mathcal{X}_j = 2} P_{j,i} \quad \forall i\in\hi\\
&p_i \geq p^{\text{(1)}}_i + \sum_{j|\mathcal{X}_j = 2}P_{j,i} + \sum_{k|\mathcal{X}_k = 1} p^{\text{(1)}}_kx_{i,\emptyset,k} \quad\forall i\in\hi \\
&\sum_{i|\mathcal{X}_i=3\cup \emptyset}\sum_{j|\mathcal{X}_j=2\cup\emptyset} x_{i,j,k} \geq 1 \quad\forall k\in \lo\\
&\sum_{i|\mathcal{X}_i=3\cup\emptyset} y_{i,j} \geq 1 \quad \forall j\in\midlvl\\
\text{where} \notag \\
&y_{i,j} \in \{0,1\} \quad\forall i \in\hi\cup\emptyset, \forall j\in \midlvl \notag\\
&x_{i,j,k}\in\{0,1\} \quad\forall i\in \hi\cup\emptyset, \forall j\in \midlvl\cup\emptyset, \notag \\ &\forall k\in \lo\cup\emptyset:  k\neq \emptyset\vee (i=\emptyset \wedge j=\emptyset) \notag\\
&p_i \in \mathbb{Z}^+_0 \quad \forall i\in\hi \notag\\
&P_{j,i} \in \mathbb{Z}^+_0 \quad 
\forall i\in\hi\cup\emptyset, \forall j\in\midlvl \notag
\end{align}

When \textit{Bottom-up} fails to prove optimality, it fallbacks to this model while supplying the $lb^-$ lower bound and the initial solution. The reason for executing \textit{Bottom-up} ahead solving MILP model~\ref{milp:mc3} is two-fold. First, we have observed the solver struggles to prove optimality when the solution is clearly optimal regarding the critical path. The other observation is that if the problem instance contains the majority of tasks with criticality one and two, then solving its $2^-$ restriction frequently yields optimal solution since the highest criticality levels are not likely to be utilized. Similar holds for the instances with a large number of tasks with higher criticality. Furthermore, solving $2^\pm$ restrictions of $I_\mathcal{MC}$ is cheap compared to the solving the whole MILP model~\ref{milp:mc3} as it can be seen in Tab.~\ref{tab:mc2}.



\section{COMPUTATIONAL EXPERIMENTS} \label{sec:results}

For the problem $1|mc=2,mu|\cmax$ we have randomly generated sets of $20$ instances with $n$ tasks for each $n\in\{10, \ldots, 200\}$. Criticalities of tasks were distributed uniformly. The processing time of a task at level $1$ is sampled from the uniform distribution $\mathcal{U}(1, 11)$. For tasks with the criticality of $2$, the prolongation at level $2$ is sampled from uniform distribution $\mathcal{U}(1, 10)$.

For the problem $1|mc=3,mu|\cmax$ we have randomly generated sets of $20$ instances with $n$ tasks for each $n\in\{10, \ldots, 80\}$. For each $n$, the set contains instances with different splits of tasks' criticalities and different distributions  for prolongation  (e.g. $\mathcal{U}(1, 10)$ and $\mathcal{U}(1, 7)$ for the second level, $\mathcal{U}(1, 10)$ and $\mathcal{U}(1, 14)$ for the third level, etc.) in order to generate instances of various properties.

\begin{table*}
\centering
\caption{Computational results for the problem $1|mc=2,mu|\cmax$.} \label{tab:mc2}
\resizebox{1\textwidth}{!}{
\begin{tabular}{| l | r | r | r | r | r | r | r | r | }
\cline{2-9}
\multicolumn{1}{c|}{} & \multicolumn{4}{c}{Covering MILP~\ref{milp:mc2}} & \multicolumn{4}{|c|}{Relative Order MILP \cite{hanzalek2016mc} } \\
\hline
$n$ tasks & avg $t$ $[s]$ & max $t$ $[s]$ & unsl $[\%]$ & avg gap $[\%]$ & avg $t$ $[s]$ & max $t$ $[s]$ & unsl $[\%]$ & avg gap $[\%]$ \\
\hline
$10$ & $>0.01$ & $0.03$ & 0 & --- &  $13.07\,(\pm 44.93)$ & $200.22$ & 0 & --- \\
$15$ & $>0.01$ & $0.03$ & 0 & --- &  $49.67\,(\pm 49.38)$ & $127.09$ & 60 & $27.32\,(\pm 12.54)$\\
$20$ & $0.01\,(\pm 0.01)$ & $0.03$ & 0 & --- &  --- & --- & 100& $40.09\,(\pm 15.27)$\\
$40$ & $0.09\,(\pm 0.17)$ & $0.81$ & 0 & --- &  --- & --- & 100& $77.66\,(\pm 6.23)$\\
$60$ & $1.37\,(\pm 4.33)$ & $19.71$ & 0 & --- &  --- & --- & 100 &$84.23\,(\pm 2.90)$ \\
80 & $0.38\,(\pm 0.45)$ & $1.94$ & 0& --- &  --- & --- & 100 & $90.72\,(\pm 1.77)$\\
100 & $1.28\,(\pm 1.38)$ & $5.05$ & 0 & --- &  --- & --- & 100 & $93.38\,(\pm 0.76)$\\
150 & $11.77\,(\pm 24.29)$ & $93.01$ & 0 & --- &  --- & --- & 100 & $96.02\,(\pm 0.24)$\\
200 & $22.69\,(\pm 61.24)$ & $281.04$ & 0 & --- &  --- & --- & 100 & $97.33\,(\pm 0.13)$\\
\hline
\end{tabular}
}
\end{table*}

\begin{table*}
\centering
\caption{Computational results for the problem $1|mc=3,mu|\cmax$.} \label{tab:mc3}
\resizebox{1\textwidth}{!}{
\begin{tabular}{| l | r | r | r | r | r | r | r | r | }
\cline{2-9}
\multicolumn{1}{c|}{} & \multicolumn{4}{c}{Bottom-up w/ Covering MILP~\ref{milp:mc3}} & \multicolumn{4}{|c|}{Relative Order MILP \cite{hanzalek2016mc}} \\
\hline
$n$ tasks & avg $t$ $[s]$ & max $t$ $[s]$ & unsl $[\%]$ & avg gap $[\%]$ & avg $t$ $[s]$ & max $t$ $[s]$ & unsl $[\%]$ & avg gap $[\%]$ \\
\hline
10 & $0.02\,(\pm 0.01)$ & $0.04$ & 0 & --- &  $0.09\,(\pm 0.07)$ & $0.28$ & 0 & --- \\
20 & $0.16\,(\pm 0.36)$ & $1.66$ & 0 & --- &  --- & --- & 100 & $28.71\,(\pm 16.62)$\\
30 & $0.17\,(\pm 0.17)$ & $0.66$ & 0 & --- &  --- & --- & 100 & $63.28\,(\pm 7.35)$ \\
40 & $0.69\,(\pm 1.02)$ & $3.61$ & 0 & --- &  --- & --- & 100 & $72.85\,(\pm 6.14)$\\
50 & $2.40\,(\pm 7.42)$ & $33.56$ & 0 & --- &  --- & --- & 100 & $80.61\,(\pm 2.96)$\\
60 & $6.71\,(\pm 11.97)$ & $44.67$ & 0 & --- &  --- & --- & 100 & $84.30\,(\pm 2.70)$\\
70 & $11.30\,(\pm 22.31)$ & $79.38$ & 10& $0.38\,(\pm 0.19)$ &  --- & --- & 100 & $89.34\,(\pm 1.43)$\\
80 & $37.92\,(\pm 68.82)$ & $224.86$ & 20& $0.34\,(\pm 0.13)$ &  --- & --- & 100 & $91.09\,(\pm 1.40)$\\
\hline
\end{tabular}
}
\end{table*}

The column \textit{avg} $t$ (\textit{max} $t$) in Tab.~\ref{tab:mc2} and \ref{tab:mc3} denotes the average (maximal) computational time for instances that were solved within the time limit of $300\jed{s}$. The column \textit{unsl} contains the percentage of instances that were not solved within the time limit and \textit{avg gap} denotes average optimality gap proven by the solver for the unsolved instances. Results were obtained with two Intel Xeon E5-2620 v2 @ $2.10\jed{GHz}$ processors using Gurobi Optimizer $6.5$ with the algorithms implemented in Python 3.4.

In Tab.~\ref{tab:mc2} it can be seen that our model is able to solve about an order of the magnitude larger problem instances. The \textit{Relative Order} model proposed by \cite{hanzalek2016mc} consistently fails to narrow optimality gap for instances with more than 40 tasks. In Tab.~\ref{tab:mc3} it is shown that the combination of \textit{Bottom-up} heuristic and MILP~\ref{milp:mc3} is able to solve reliably instances with 60 tasks up to the optimality and almost all instances with 80 tasks. Moreover, the proven gap is much smaller than for the \textit{Relative Order} model; therefore it shows that our model has stronger linear relaxation.

\section{CONCLUSION} \label{sec:conclusion}
In this paper, we have proposed two exact approaches for the problem of non-preemptive mixed-criticality match-up scheduling for solving the problem of message retransmission in time-triggered communication protocols. The algorithms outperform the approach recently proposed by a large margin. Furthermore, we showed the membership of $1|mc=\mathcal{L},mu|\cmax$ problem in $\mathcal{APX}$ complexity class for an arbitrary fixed $\mathcal{L}$. 
\bibliographystyle{plain}
{\small
\bibliography{references}}

\begin{thebibliography}{10}

\bibitem{allahverdi2015third}
Ali Allahverdi.
\newblock The third comprehensive survey on scheduling problems with setup
  times/costs.
\newblock {\em European Journal of Operational Research}, 246(2):345--378,
  2015.

\bibitem{baruah2011certification}
Sanjoy Baruah and Gerhard Fohler.
\newblock Certification-cognizant time-triggered scheduling of
  mixed-criticality systems.
\newblock In {\em Real-Time Systems Symposium (RTSS), 2011 IEEE 32nd}, pages
  3--12. IEEE, 2011.

\bibitem{baruah2010towards}
Sanjoy Baruah, Haohan Li, and Leen Stougie.
\newblock Towards the design of certifiable mixed-criticality systems.
\newblock In {\em Real-Time and Embedded Technology and Applications Symposium
  (RTAS), 2010 16th IEEE}, pages 13--22. IEEE, 2010.

\bibitem{bean1991matchup}
James~C Bean, John~R Birge, John Mittenthal, and Charles~E Noon.
\newblock Matchup scheduling with multiple resources, release dates and
  disruptions.
\newblock {\em Operations Research}, 39(3):470--483, 1991.

\bibitem{bell2006introduction}
Ron Bell.
\newblock Introduction to iec 61508.
\newblock In {\em Proceedings of the 10th Australian workshop on Safety
  critical systems and software-Volume 55}, pages 3--12. Australian Computer
  Society, Inc., 2006.

\bibitem{bertsimas2011theory}
Dimitris Bertsimas, David~B Brown, and Constantine Caramanis.
\newblock Theory and applications of robust optimization.
\newblock {\em SIAM review}, 53(3):464--501, 2011.

\bibitem{burns2013mixed}
Alan Burns and Rob Davis.
\newblock Mixed criticality systems-a review.
\newblock {\em Department of Computer Science, University of York, Tech. Rep},
  2013.

\bibitem{dvorak2016using}
J.~Dvorak and Z.~Hanzalek.
\newblock Using two independent channels with gateway for {F}lex{R}ay static
  segment scheduling.
\newblock {\em IEEE Transactions on Industrial Informatics}, article in press,
  2016.

\bibitem{hanzalek2016mc}
Zdenek Hanzalek, Tomas Tunys, and Premysl Sucha.
\newblock Non-preemptive mixed-criticality match-up scheduling problem.
\newblock {\em Journal of Scheduling, doi: 10.1007/s10951-016-0468-y}, 2016.

\bibitem{kopetz2005time}
Hermann Kopetz, Astrit Ademaj, Petr Grillinger, and Klaus Steinhammer.
\newblock The time-triggered ethernet (tte) design.
\newblock In {\em Object-Oriented Real-Time Distributed Computing, 2005. ISORC
  2005. Eighth IEEE International Symposium on}, pages 22--33. IEEE, 2005.

\bibitem{qi2006disruption}
Xiangtong Qi, Jonathan~F Bard, and Gang Yu.
\newblock Disruption management for machine scheduling: the case of spt
  schedules.
\newblock {\em International Journal of Production Economics}, 103(1):166--184,
  2006.

\bibitem{sahinidis2004optimization}
Nikolaos~V Sahinidis.
\newblock Optimization under uncertainty: state-of-the-art and opportunities.
\newblock {\em Computers \& Chemical Engineering}, 28(6):971--983, 2004.

\bibitem{shabtay2013survey}
Dvir Shabtay, Nufar Gaspar, and Moshe Kaspi.
\newblock A survey on offline scheduling with rejection.
\newblock {\em Journal of scheduling}, 16(1):3--28, 2013.

\bibitem{theis2013schedule}
Jens Theis, Gerhard Fohler, and Sanjoy Baruah.
\newblock Schedule table generation for time-triggered mixed criticality
  systems.
\newblock {\em Proc. WMC, RTSS}, pages 79--84, 2013.

\bibitem{vestal2007preemptive}
Steve Vestal.
\newblock Preemptive scheduling of multi-criticality systems with varying
  degrees of execution time assurance.
\newblock In {\em Real-Time Systems Symposium, 2007. RTSS 2007. 28th IEEE
  International}, pages 239--243. IEEE, 2007.

\end{thebibliography}

\end{document}